\documentclass[a4paper,10pt]{article}

\usepackage[T1]{fontenc} 
\usepackage{amsmath,amsthm}
\usepackage[english]{babel}
\usepackage[dvips]{graphicx}
\usepackage{amsfonts,amssymb}
\usepackage{geometry}
\usepackage{hyperref}
\usepackage{stmaryrd}
\usepackage{fancyhdr}
\usepackage{url}
\usepackage{dsfont}

\newtheorem*{conj*}{Conjecture}

\newtheorem*{thm*}{Theorem}

\newtheorem{prop}{Proposition}[section]

\newtheorem{LM}{Lemma}[section]

\newtheorem{thm}{Theorem}[section]

\newtheoremstyle{pourlesremarques}{\topsep}{\topsep}{\normalfont}{}{\bfseries}{.}{ }{}
\theoremstyle{pourlesremarques}

\newtheorem*{rem*}{Remark}
\newtheoremstyle{pourlesexemples}{\topsep}{\topsep}{\normalfont}{}{\bfseries}{.}{ }{}
\theoremstyle{pourlesexemples}

\renewcommand{\l}{\lambda}

\newcommand{\sm}{\mathcal{C}^\infty}

\title {\textbf{Cuspidal representations of $GL(n,F)$ distinguished by a maximal Levi subgroup,
 with $F$ a non-archimedean local field}}
\author{Nadir MATRINGE\footnote{Nadir Matringe, Universit\'e de Poitiers, Laboratoire de Math\'ematiques et Applications,
T\'el\'eport 2 - BP 30179, Boulevard Marie et Pierre Curie, 86962, Futuroscope Chasseneuil Cedex. Email: Nadir.Matringe@math.univ-poitiers.fr}}

\begin{document}
\maketitle

\begin{abstract}
Let $\rho$ is a cuspidal representation of $GL(n,F)$, with $F$ 
a non archimedean local field, and $H$ a maximal Levi subgroup of $GL(n,F)$. 
We show that if $\rho$ is $H$-distinguished, then $n$ is even, and $H\simeq GL(n/2,F)\times GL(n/2,F)$.
\end{abstract}

\section{Preliminaries}

Let $F$ be nonarchimedean local field. We denote $GL(n,F)$ by $G_n$ for $n\geq 1$, and by $N_n$ the unipotent radical of the Borel subgroup of $G_n$ given by upper triangular matrices. 
For $n\geq2$ we denote by $U_n$ the group of matrices $u(x)=\begin{pmatrix} 
                                                                                         I_{n-1}    & x\\
                                                                                                    & 1 \end{pmatrix}$ 
                                                                                                    for $x$ in $F^{n-1}$.\\

For $n> 1$, the map $g\mapsto \begin{pmatrix} g & \\ & 1 \end{pmatrix}$ is an embedding of the group $G_{n-1}$ in $G_{n}$, we denote by $P_n$ the subgroup $G_{n-1}U_n$ of $G_n$.\\
We fix a nontrivial character $\theta$ of $(F,+)$, and denote by $\theta$ again the character $n\mapsto \theta(\sum_{i=1}^{n-1}n_{i,i+1})$ of $N_n$.
The normaliser of $\theta_{|U_n}$ is then $P_{n-1}$.\\

When $G$ is an $l$-group (locally compact totally disconnected group), we denote by $Alg(G)$ the category of smooth complex $G$-modules. If $(\pi,V)$ belongs to $Alg(G)$, $H$ is a closed subgroup of $G$,
 and $\chi$ is a character of $H$, %we denote by $V(H,\chi)$ the subspace of $V$ generated by vectors of the form $\pi(h)v-\chi(h)v$ for $h$ in $H$ and $v$ in $V$. 
%This space is actually stable under the action of the subgroup $N_G(\chi)$ of the normalizer $N_G(H)$ of $H$ in $G$, which fixes $\chi$.\\
we denote by $\delta_H$ the positive character of $N_G(H)$ such that if $\mu$ is a right Haar measure on $H$, and $int$ is the action given by $(int(n)f)(h)=f(n^{-1}hn)$, of $N_G(H)$ smooth functions $f$ with compact support on $H$, then $\mu \circ int(n)= \delta_H(n)\mu $ for $n$ in $N_G(H)$.\\ 
%The space $V(H,\chi)$ is $N_G(\chi)$-stable. Thus, if $L$ is a closed-subgroup of $N_G(\chi)$, and $\mu$ is a (smooth) character of $L$, the quotient $V_{H,\chi}=V/V(H,\chi)$ (that we simply denote by $V_H$ when 
%$\chi$ is trivial) becomes a smooth $L$-module for the action $l.(v + V(H,\chi))= \mu(l)\pi(l)v + V(H,\chi)$ of $L$ on 
%$V_{H,\chi}$.\\  
If $H$ is a closed subgroup of an $l$-group $G$, and $(\rho,W)$ belongs to $Alg(H)$, we define the object 
$(ind_H^G(\rho), V_c=ind_H^G(W))$ %and $(Ind_H^G(\rho), V=Ind_H^G(W))$ of $Alg(G)$ 
as follows.
 The space $V_c$ is the space of smooth functions from $G$ to $W$, fixed under right translation by the elements of a compact open subgroup 
$U_f$ of $G$, satisfying $f(hg)=\rho(h)f(g)$ for all $h$ in $H$ and $g$ in $G$,
%. The space $V_c$ is the subspace of $V$, consisting of functions 
and with support compact mod $H$. The action of $G$ is by right translation on the functions.\\
If $f$ is a function from $G$ to another set, and $g$ belongs to $G$, we will denote $L(g)f:x\mapsto f(g^{-1}x)$ and 
$R(g)f:x\mapsto f(xg)$.\\

We say that a representation $\pi$ of $G$ is $H$-distinguished, if the complex vector space $Hom_H(\pi,1)$ is nonzero.\\

We will use the following functors following \cite{BZ2}:\\

\begin{itemize}

 %\item The functor $\Phi^{-}$ from $Alg(P_k)$ to $Alg(P_{k-1})$ such that, if $(\pi,V)$ is a smooth $P_k$-module, 
%$\Phi^{-} V =V_{U_k,\theta}$, and $P_{k-1}$ acts on $\Phi^{-}(V)$ by
% $\Phi^{-} \pi (p)(v+V(U_k,\theta))= \delta_{U_k} (p)^{-1/2}\pi (p)(v+V(U_k,\theta))$.

\item The functor $\Phi^{+}$ from $Alg(P_{k-1})$ to $Alg(P_{k})$ such that, for $\pi$ in $Alg(P_{k-1})$, one has
$\Phi^{+} \pi = ind_{P_{k-1}U_k}^{P_k}(\delta_{U_k}^{1/2}\pi \otimes \theta)$.

%\item The functor $\hat{\Phi}^{+}$ from $Alg(P_{k-1})$ to $Alg(P_{k})$ such that, for $\pi$ in $Alg(P_{k-1})$, one has
%$\hat{\Phi^{+}} \pi = Ind_{P_{k-1}U_k}^{P_k}(\delta_{U_k}^{1/2}\pi \otimes \theta)$.

% \item The functor $\Psi^{-}$ from $Alg(P_k)$ to $Alg(G_{k-1})$, such that if $(\pi,V)$ is a smooth $P_k$-module, 
% $\Psi^{-} V =V_{U_k,1}$, and $G_{k-1}$ acts on $\Psi^{-}(V)$ by
% $\Psi^{-} \pi (g)(v)+V(U_k,1)= \delta_{U_k} (g)^{-1/2}\pi (p)(v+V(U_k,1))$.

\item The functor $\Psi^{+}$ from $Alg(G_{k-1})$ to $Alg(P_{k})$, such that for $\pi$ in $Alg(G_{k-1})$, one has
$\Psi^{+} \pi = ind_{G_{k-1}U_k}^{P_k}(\delta_{U_k}^{1/2}\pi \otimes 1)=\delta_{U_k}^{1/2}\pi \otimes 1 $.

\end{itemize}
 
We recall the following proposition, which is a consequence of theorem 4.4 of \cite{BZ2}.

\begin{prop}\label{kir}
Let $\pi$ be a cuspidal representation of $G_n$, then the restriction $\pi_{|P_n}$ is isomorphic to $(\Phi^+)^{n-1}\Psi^+(1)$.
\end{prop}

\section{The result}

Suppose $n=p+q$, with $p\geq q\geq 1$, we denote by $M_{(p,q)}$ the standard Levi of $G_n$ given by matrices 
$\begin{pmatrix} h_p &  \\  & h_q  \end{pmatrix}$ with $h_p\in G_p$ and $h_q\in G_q$, and by 
$M_{(p,q-1)}$ the standard Levi of $G_{n-1}$ given by matrices 
$\begin{pmatrix} h_p &  \\  & h_{q-1}  \end{pmatrix}$ with $h_p\in G_p$ and $h_{q-1}\in G_{q-1}$. 
We denote by $M_{(p-1,q-1)}$ the standard Levi of $G_{n-2}$ given by matrices 
$\begin{pmatrix} h_{p-1} &  \\  & h_{q-1}  \end{pmatrix}$ with $h_{p-1}\in G_{p-1}$ and $h_{q-1}\in G_{q-1}$.\\
Let $w_{p,q}$ be the permutation matrix of $G_n$ corresponding to the permutation 
$$\left(\begin{array}{llllllllllllllll} 
1 & \!\! \dots \!\! & p-q & p-q +1 & p-q +2 &\!\! \dots \!\! & p-1    & p      & p+1      &\!\! \dots \!\!  & p+q-2 & p+q-1 & p+q \\
1 & \!\! \dots \!\! & p-q & p-q +1 & p-q +3 &\!\!  \dots \!\!& p+q-3 & p+q-1 & p-q + 2 &\!\! \dots \!\!& p+q-4 & p+q-2 & p+q
\end{array}\right)$$
Let $w_{p,q-1}$ be the permutation matrix of $G_{n-1}$ corresponding to the permutation $w_{p,q}$ restricted to $\{1,\dots,n-1\}$:
$$\left(\begin{array}{lllllllllllllll} 
1 & \!\! \dots \!\! & p-q & p-q +1 & p-q +2 &\!\! \dots \!\! & p-1    & p      & p+1      &\!\! \dots \!\!  & p+q-2 & p+q-1  \\
1 & \!\! \dots \!\! & p-q & p-q +1 & p-q +3 &\!\!  \dots \!\!& p+q-3 & p+q-1 & p-q + 2 &\!\! \dots \!\!& p+q-4 & p+q-2 
\end{array}\right)$$ 
Let $w_{p-1,q-1}$ be the permutation matrix of $G_{n-2}$ corresponding to the permutation
$$\left(\begin{array}{lllllllllllllll} 
1 & \!\! \dots \!\! & p-q & p-q+1  & p-q +2 &\!\! \dots \!\! & p-2   & p-1      & p        &\!\! \dots \!\!&  p+q-3 & p+q-2   \\
1 & \!\! \dots \!\! & p-q & p-q+1  & p-q +3 &\!\!  \dots \!\!& p+q-5 & p+q-3   & p-q + 2  &\!\! \dots \!\!&  p+q-4 & p+q-2 
\end{array}\right)$$ 

We denote by $H_{p,q}$ the subgroup $w_{p,q}M_{(p,q)}w_{p,q}^{-1}$ of $G_n$, by $H_{p,q-1}$ the subgroup $w_{p,q-1}M_{(p,q-1)}w_{p,q-1}^{-1}$ of $G_{n-1}$, and by $H_{p-1,q-1}$ the subgroup $w_{p-1,q-1}M_{(p-1,q-1)}w_{p-1,q-1}^{-1}$ of $G_{n-2}$.\\

The two following lemmas and propositions are a straightforward adaptation of Lemma 1 and Proposition 1 of \cite{K}.

\begin{LM}\label{lm1}
Let $S_{p,q}=\{g\in G_{n-1}, \forall u\in U_n\cap H_{p,q}, \ \theta(gug^{-1})=1\}$. Then $S_{p,q}=P_{n-1}H_{p,q-1}$.
\end{LM}
\begin{proof}
Denoting by $L_{n-1}(g)$ the bottom row of $g$, one has $\theta(g u(x)g^{-1})=\theta (L_{n-1}(g).x)$ for $u(x)$ in $U_n$. Hence $\theta(gug^{-1})=1$ for all $u$ in  $U_n\cap H_{p,q}$ if and only if 
$g_{n-1,j}=0$ for $j=p-q,p-q+2,\dots,p+q-2$. It is equivalent to say that $g$ belongs to $P_{n-1}H_{p,q-1}$.
\end{proof}

\begin{LM}\label{lm2}
Let $S_{p,q-1}=\{g\in G_{n-2}, \forall u\in U_{n-1}\cap H_{p,q-1}, \ \theta(g^{-1}ug)=1\}$. 
Then $S_{p,q}=P_{n-2}H_{p-1,q-1}$.
\end{LM}
\begin{proof} Denoting by $L_{n-2}(g)$ the bottom row of $g$, and by $u(x)$ the matrix $\begin{pmatrix} I_{n-2}& x \\ 0 & 1 \end{pmatrix}$, 
sothat $\theta(gug^{-1})=\theta (L_{n-2}(g).x)$. Hence $\theta(gug^{-1})=1$ for all $u$ in  $U_{n-1}\cap H_{p,q-1}$ if and only if 
$g_{n-2,j}=0$ for $j=0,1,\dots,p-q,p-q+1$ and $j=p-q+3,p-q+5\dots,p+q-5,p+q-3$. It is equivalent to say that $g$ belongs to $P_{n-2}H_{p-1,q-1}$.
\end{proof}

\begin{prop}\label{prop1}
Let $\sigma$ belong to $Alg(P_{n-1})$, and $\chi$ be a positive character of $P_{n}\cap H_{p,q}$, then there is a positive character $\chi'$ of 
$P_{n-1}\cap H_{p,q-1}$, such that $$Hom_{P_n\cap H_{p,q}}(\Phi^+ \sigma,\chi)\hookrightarrow Hom_{P_{n-1}\cap H_{p,q-1}}(\sigma,\chi').$$
\end{prop}
\begin{proof}
 Let $V$ be the space on which $\sigma$ acts, and $W=\phi^+ V$. 
Let $A$ the projection from $\sm_c(P_n,V)$ onto $W$, defined by 
$A(f(p))= \int_{P_{n-1}U_n} \delta_{U_n}^{-1/2}(y)\sigma(y^{-1})f(yg) dy$. Lifting through $A$ gives a vector space injection of 
 $Hom_{P_n\cap H_{p,q}}(\Phi^+ \sigma,\chi)$ into the space of $V$-distributions $T$ on $P_n$ 
satisfying relations 
\begin{eqnarray}\label{eq1}
T\circ R(h_0) = \chi(h_0)T
\end{eqnarray}
\begin{eqnarray}\label{eq2}
T\circ L(y_0) = \delta_{U_n}^{3/2}(y_0) T\circ \sigma(y_0)
\end{eqnarray}
for $h_0$ in $P_n\cap H_{p,q}$ and $y_0\in P_{n-1}U_n$.\\
We introduce $\Theta$ the map on $P_n$ defined by $\Theta (ug)=\theta(u)$ for $u$ in $U_n$ and $g$ in $G_{n-1}$.
Then the $V$-distribution $\Theta. T$ is $U_n$-invariant, hence there is a $V$-distribution $S$ with support in 
$G_{n-1}$ such that $\Theta.T= du\otimes S$ (where $du$ denotes a Haar measure on $U_n$), and thus
$T=\Theta^{-1}.du\otimes S$ has support $U_n. \text{supp}(S)$. It is easely verified that $du\otimes S$ is 
invariant-$U_n$, but because of relation (\ref{eq1}), $T$ is invariant-$(U_n\cap H_{p,q})$. We deduce from these two facts that for $g$ in $\text{supp}(S)$, 
$\Theta(gu)$ must be equal to $\Theta(g)$ for any $u$ in $U_n\cap H_{p,q}$. This means that $\text{supp}(S)\subset S_{p,q}$, and $S_{p,q}=P_{n-1}H_{p,q-1}$ according to Lemma \ref{lm1}, hence $T$ has support in $P_{n-1}U_nH_{p,q-1}$.\\
Now consider the projection $B:\sm_c(P_{n-1}U_n\times H_{p,q-1},V)\twoheadrightarrow \sm_c(P_{n-1}U_nH_{p,q-1},V)$, defined by $B(\phi)(y^{-1}h)=\int_{P_{n-1}\cap H_{p,q-1}} \phi(ay,ah)da$ (which is well defined because of the equality 
$P_{n-1}U_n\cap H_{p,q-1}= P_{n-1}\cap H_{p,q-1}$), and $\phi \mapsto \tilde{\phi}$ the isomorphism of $\sm_c(P_{n-1}U_n \times H_{p,q-1},V)$ defined by $\tilde{\phi}(y,h)=\chi(h)\delta_{U_n}(y)^{3/2}\sigma(y)\phi(y,h)$.\\
If one sets $D(\phi)=T(B(\tilde{\phi}))$, then $D$ is a $V$-distribution on $P_{n-1}U_n\times H_{p,q-1}$ which is invariant-$P_{n-1}U_n\times H_{p,q-1}$. This implies that there exists a unique linear form $\lambda$ on $V$, such that 
for all $D(\phi)=\int_{P_{n-1}U_n\times H_{p,q-1}} \lambda(\phi(y,h)) dy dh$.\\
Now for $b$ in $P_{n-1}\cap H_{p,q-1}$, on has from the integral expression of $D$, the relation 
$D\circ L(b,b)= \delta(b) D$ for some positive modulus character $\delta$. On the other hand, writing $D$ as 
$\phi\mapsto T(B(\tilde{\phi}))$, one has 
$\widetilde{L(b,b)\phi}=\chi(b)\delta_{U_n}^{3/2}(b)L(b,b)(\widetilde{\sigma(b^{-1})\phi})$ and $B\circ L(b,b)=\delta_1(b) B$ for a positive modulus character $\delta_1$, sothat 
$D\circ L(b,b)= \delta_1(b)\chi(b)\delta_{U_n}^{3/2}(b)D\circ \sigma(b^{-1})$. Comparing the two expressions for $D\circ L(b,b)$, 
we get the relation  $D\circ \sigma(b)= \chi'(b)D$, with $\chi'$ being the positive character $\delta^{-1}\delta_1\chi\delta_{U_n}^{3/2}$ of $P_{n-1}\cap H_{p,q-1}$.\\
This in turn implies that the linear form $\lambda$ on $V$ satisfies the same relation, i.e. belongs 
to $Hom_{P_{n-1}\cap H_{p,q-1}}(\sigma,\chi')$, and $T\mapsto \l$ gives a linear injection of $Hom_{P_n\cap H_{p,q}}(\Phi^+ \sigma,\chi)$ into $Hom_{P_{n-1}\cap H_{p,q-1}}(\sigma,\chi')$, and this proves the proposition.

\end{proof}

Using Lemma \ref{lm2} instead of Lemma \ref{lm1} in the previous proof, one obtains the following statement.

\begin{prop}\label{prop2}
Let $\sigma'$ belong to $Alg(P_{n-2})$, and $\chi'$ be a positive character of $P_{n-1}\cap H_{p,q-1}$, then there is a positive character $\chi''$ of 
$P_{n-2}\cap H_{p-1,q-1}$, such that $$Hom_{P_{n-1}\cap H_{p,q-1}}(\Phi^+ \sigma',\chi')\hookrightarrow Hom_{P_{n-2}\cap H_{p-1,q-1}}(\sigma,\chi'').$$
\end{prop}

A consequence of these two propositions is the following.

\begin{prop}\label{homzero}
Let $n\geq 3$, and $p$ and $q$ two integers with $p+q=n$ and $p-1\geq q\geq 0$, then one has $Hom_{P_n\cap H_{p,q}}((\Phi^+)^{n-1}\Psi^+(1),1)=0$.
\end{prop}
\begin{proof} Using repeatedly the last two propositions, we get the existence of a positive character $\chi$ of $P_{p-q+1}$ such that 
$Hom_{P_n\cap H_{p,q}}((\Phi^+)^{n-1}\Psi^+(1),1)\hookrightarrow Hom_{P_{p-q+1}\cap H_{p-q+1,0}}((\Phi^+)^{p-q}\Psi^+(1),\chi)=Hom_{P_{p-q+1}}((\Phi^+)^{p-q}\Psi^+(1),\chi)$, and this last space is $0$ 
because $(\Phi^+)^{p-q}\Psi^+(1)$ and $\chi$ are two non isomorphic irreducible representations of $P_{p-q+1}$, according to corollary 3.5 of \cite{BZ2}.

\end{proof}

This implies the following theorem about cuspidal representations.

\begin{thm}
Let $\pi$ be a cuspidal representation of $G_n$, which is distinguished by a maximal Levi subgroup $M$, 
then $n$ is even and $M\simeq M_{n/2,n/2}$.
\end{thm}
\begin{proof} Let $M$ be the maximal Levi subgroup such that $\pi$ is $M$-distinguished. Then $M$ is conjugate to a standard Levi subgroup 
$M_{p,q}$ with $p\geq q$ and $p+q=n$. Suppose $p\geq q+1$, $M_{p,q}$ is conjugate to $H_{p,q}$, sothat $\pi$ is $H_{p,q}$-distinguished, and 
$\pi_{|P_n}$ is thus $H_{p,q}\cap P_n$-distinguished.  
But by Proposition \ref{kir}, the restriction $\pi_{|P_n}$ is isomorphic to $(\Phi^+)^{n-1}\Psi^+(1)$, and this contradicts 
Proposition \ref{homzero}. Hence one must have $p=q$, and this proves the theorem.\end{proof}

\end{document}